\documentclass[12pt]{article}

\usepackage{amsmath,amssymb,amsfonts}
\usepackage{amsthm}%
\usepackage{mathrsfs}
\usepackage{enumerate}
\usepackage{graphicx}
\usepackage{stmaryrd}
\usepackage[OT1]{fontenc}

\allowdisplaybreaks%

{\theoremstyle{plain}%
 \newtheorem{theorem}{Theorem}

}
{\theoremstyle{remark}

}
{\theoremstyle{definition}

}

\begin{document}

\begin{center}
 {\large On two conjectures on triangulations of 2-manifolds}

 \ 

{\sc John M. Campbell}

\vspace{0.1in}

{\footnotesize Department of Mathematics and Statistics}

{\footnotesize Dalhousie University}

{\footnotesize Halifax, NS B3H 4R2}

{\footnotesize Canada}

\vspace{0.1in}

{\footnotesize {\tt jh241966@dal.ca}}

\vspace{0.1in}

\end{center}

\begin{abstract}
 For a closed, connected 
 2-manifold $M$, and for a triangulation $T$ of $M$, we write $V(T)$ in place of the vertex set associated with $T$. A 
 \emph{cyclic coloration} of a triangulation $T$ of $M$ refers to a face coloring of $T$ such that: For each $v \in V(T)$, the faces incident 
 to $v$ have distinct colors. Chen and Lawrencenko [\emph{Yokohama Math.\ J.}, 1999] conjectured that there exists a constant $C(M)$ 
 (depending only on $M$) such that $|V(T)| + C(M)$ colors suffice for there to exist a cyclic coloration of a triangulation $T$ of $M$. We 
 prove this conjecture, using a greedy algorithm related to the Euler--Poincar\'{e} formula for surface triangulations. Chen and 
 Lawrencenko also conjectured that: If $M$ is not the projective plane and $T$ is a triangulation of $M$ that is minimal with respect to 
 the number of vertices, then $\xi(T) = |V(T)| = V_{\min}(M)$, where $V_{\min}(M)$ denotes the minimum possible number of vertices 
 among all triangulations of the 2-manifold $M$, and where $\xi(T)$ denotes the minimal value $k$ such that $T$ admits a cyclic coloration 
 with $k$ colors. We disprove this latter conjecture via an explicit counterexample, using an 8-vertex triangulation of the Klein bottle with 
 16 faces. It appears that both of the Chen--Lawrencenko conjectures have remained open, prior to our work. 
\end{abstract}

\vspace{0.1in}


\vspace{0.1in}


\section{Introduction}
 A main area of study in combinatorial topology is given by triangulating manifolds. Landmark results in this area include Milnor's 1961 
 disproof of the polyhedral Hauptvermutung for finite simplicial complexes \cite{Milnor1961}, and a major twenty-first century 
 development is given by Manolescu's disproof of the high-dimensional Triangulation Conjecture, which, in conjunction with prior work 
 by Galewski--Stern \cite{GalewskiStern1980} and Matumoto \cite{Matumoto1978}, showed the existence of nontriangulable topological 
 manifolds in every dimension at least $5$ \cite{Manolescu2016}. The purpose of this paper is to solve two open problems that concern 
 manifold triangulations and that were introduced in 1999 by Chen and Lawrencenko~\cite{ChenLawrencenko1999}. It appears that these 
 two conjectures have remained open, prior to our paper. 

 As in the work of Chen and Lawrencenko~\cite{ChenLawrencenko1999}, the manifolds considered in this paper are connected, closed 
 2-manifolds. For a manifold $M$ of this form, and for a triangulation $T$ on $M$, we write $V(T)$ in place of the vertex set of $T$. 
 According to Chen and Lawrencenko, a face coloring is referred to as a \emph{cyclic coloration} if: For each $v \in V(T)$, the faces 
 incident to $v$ are colored differently. Chen and Lawrencenko's cited work is largely based on two conjectures that they introduced and 
 that are reproduced below. 

 Again letting $M$ be a closed, connected 2-manifold, Chen and Lawrencenko~\cite{ChenLawrencenko1999} conjectured that there 
 exists a constant $C(M)$, depending only on $M$, whereby $|V(T)| + C(M)$ colors are sufficient for there to be a cyclic coloration for an 
 arbitrary triangulation $T$ of $M$. This is given as Conjecture 2 in Chen and Lawrencenko's paper. We refer to this conjecture as 
 \emph{Chen and Lawrencenko's second conjecture}. 
 
 Being consistent with Chen and Lawrencenko's notation \cite{ChenLawrencenko1999}, we write $V_{\min}(M)$ in place of the minimum 
 number of vertices possible for a triangulation of $M$. Borrowing from Chen and Lawrencenko's terminology, a triangulation $T$ of $M$ 
 is said to be a \emph{minimal 
 triangulation} if $|V(T)| = V_{\min}(M)$, and the \emph{cyclic chromatic number} $\xi(T)$ is the least number of colors sufficient for there 
 to be a cyclic coloration of $T$. Chen and Lawrencenko conjectured that: 
 If $M$ is not the projective plane, 
 then $\xi(T) = |V(T)| = V_{\min}(M)$. 
 This is given as Conjecture 1 in their paper, and we refer to this conjecture as \emph{Chen and Lawrencenko's first conjecture}. 
 We succeed, based on our extensive interactions with GPT-5.6 Pro, in proving and disproving, respectively, 
 Chen and Lawrencenko's second and first conjectures. 
 
 Cyclic colorations of a triangulation $T$ are precisely, under the canonical bijection such that $F(T) \cong V(T^{\ast})$, cyclic vertex 
 colorings of the cellular dual $T^{\ast}$. A result from the work of Enomoto and Hor\v n\'ak \cite{EnomotoHornak2009} on cyclic vertex 
 colorings implies (in an equivalent way) that $\xi(T) \leq |V(T)| + 4$ specifically for a spherical triangulation $T$, 
 whereas Conjecture 2 from Chen and Lawrencenko's paper \cite{ChenLawrencenko1999} 
 concerns (as above) bounding $\xi(T)$ for triangulations on \emph{arbitrary} connected, closed 2-manifolds. 
 This illustrates the interest in our techniques and results in relation to extant work on cyclic vertex colorings. 

\section{Preliminaries}\label{secPrelim}
 While we assume some basic familiarity with manifolds, fundamental polygons, triangulating manifolds, etc., 
 the below preliminaries are necessary for our purposes. 

\subsection{Basics}
 A \emph{finite abstract simplicial complex} is a tuple 
 $T = (V(T), \mathcal{S}(T))$, 
 for a finite set $V = V(T)$ referred to as its \emph{vertex set}, 
 and for a collection $\mathcal{S} = \mathcal{S}(T)$ 
 of finite subsets of $V$, referred to as the \emph{simplices} of $V$, 
 such that: (a)
 For each $v \in V$, the set $\{ v \}$ is in $\mathcal{S}$; and 
 (b) If $\sigma \in \mathcal{S}$ and $\tau \subseteq \sigma$, then $\tau \in \mathcal{S}$. 
 For $\sigma \in \mathcal{S}(T)$, if $|\sigma| = k + 1$, then $\sigma$ is referred to as a \emph{$k$-simplex}. 
 It is common to abuse notation by letting it be understood that an element in $T$ refers to an element in $\mathcal{S}$. 
 For a simplex $\sigma \in T$, this element is said to have 
 dimension equal to $\dim \sigma = |\sigma| - 1$. 
 The 1-simplices are referred to as \emph{edges}, and the 2-simplices are referred to as (triangular) \emph{faces}. 

 Again let $T = (V(T), \mathcal{S}(T))$ be a finite abstract simplicial complex. A simplex $\sigma \in \mathcal{S}(T)$ is said to be a 
 \emph{facet} of $T$ if it is maximal with respect to inclusion among the simplices of $T$, i.e., 
 so that $\sigma \subseteq \tau \in \mathcal{S}(T) \Longrightarrow \tau = \sigma$. 
 Observe that a finite abstract simplicial complex is entirely/uniquely determined by its vertex set and its sets of facets. 
 The simplicial complex $T$ is said to be a \emph{pure 2-dimensional simplicial complex}
 if every facet of $T$ is a $2$-simplex. 

 The \emph{edge set} $E(T)$ of $T$ is the set of all 1-simplices of $T$. The \emph{face set} $F(T)$ of $T$ is the set of all 2-simplices
 of $T$. 
 For the case whereby the finite simplicial complex $T$ is of dimension at most $2$, define 
\begin{equation}\label{EulerPoincare}
 \chi(T) = |V(T)| - |E(T)| + |F(T)|. 
\end{equation}
 If $T$ is a triangulation of $M$ (a 2-manifold), 
 then we write $\chi(T) = \chi(M)$. This is referred to as the \emph{Euler--Poincar\'{e} characteristic} of $M$. 

 Again for a finite abstract simplicial complex $T = (V(T), \mathcal{S}(T))$ and for $v \in V(T)$, the \emph{link} of $v$ in $T$
 refers to the subcomplex
\begin{equation}\label{definelk} 
 \operatorname{lk}_{T}(v) := \{ \sigma \in T : v \not\in \sigma \ \text{and} \ \sigma \cup \{ v \} \in T \}. 
\end{equation}
 The \emph{1-skeleton} of $T$ is denoted with $T^{(1)}$
 and may be defined so that 
\begin{equation}\label{defineskeleton}
 T^{(1)} := \{ \sigma \in T : \dim \sigma \leq 1 \}. 
\end{equation}

\subsection{Geometric realizations}
 Again let $T = \big( V(T), \mathcal{S}(T) \big)$ be a finite abstract simplicial complex. 
 We regard $\mathbb{R}^{V(T)}$ as the real vector space consisting of functions $x\colon V(T) \to \mathbb{R}$.
 For $v \in V(T)$, we write $e_{v} \in \mathbb{R}^{V(T)}$
 in place of the indicator function such that $e_{v}(w) = 1$ if $w = v$ and such that 
 $e_{v}(w)$ vanishes otherwise. For $\sigma \in \mathcal{S}(T)$, its \emph{geometric realization} is 
 $$ \left| \sigma \right| 
 = \left\{ \sum_{v \in \sigma} \lambda_{v} e_{v} : \lambda_{v} \geq 0 \ \text{for every} \ v \in \sigma, 
 \sum_{v \in \sigma} \lambda_{v} = 1 \right\}. $$
 Similarly, the \emph{geometric realization} of $T$ is the topological space
 $$ |T| := \bigcup_{\substack{\sigma \in \mathcal{S}(T) \\ \sigma \neq \varnothing }} 
 \left| \sigma \right| \subseteq \mathbb{R}^{V(T)} $$
 endowed with the subspace topology
 inherited from $\mathbb{R}^{V(T)}$ as a Euclidean space. 

\subsection{Related notions}
 A \emph{simplicial circle} is a 1-dimensional and finite abstract simplicial complex $C$ such that the geometric
 realization of $C$ is homeomorphic to $S^1$. Equivalently, the simplicial complex $C$ is isomorphic to the 
 boundary complex of a polygon. 

 By again letting $M$ denote a connected, closed 2-manifold, a 
 \emph{triangulation} $T$ of $M$
 is understood, for the purposes of this paper, as a finite simplicial triangulation of $M$, i.e., 
 so that $T$
 is a finite simplicial complex such that the geometric realization of $T$ is homeomorphic to $M$. 
 It should be emphasized that generalized and ideal triangulations are not considered in this paper. 
 For a finite abstract simplicial complex $D = (V(D), \mathcal{S}(D))$, 
 this is said to be a \emph{simplicial $2$-disk} if $D$ is a pure $2$-dimensional simplicial complex whose geometric realization is 
 homeomorphic to the closed unit disk 
 $\mathbb{D}^{2} = \{ x \in \mathbb{R}^{2} : \| x \| \leq 1 \}$. 

\section{Proof of Chen and Lawrencenko's second conjecture}\label{secproof}
 Our proof of the second Chen--Lawrencenko conjecture \cite[Conjecture 2]{ChenLawrencenko1999} provides a 
 stronger result, in the sense that we show how the required constant $C(M)$ 
 can be chosen to depend only on the Euler--Poincar\'{e} characteristic $\chi(M)$. 

\begin{theorem}\label{firstsolution}
 For every connected, closed 2-manifold $M$, 
 there exists a constant $C(M)$ such that: For every finite simplicial triangulation $T$ of $M$, 
 the relation 
\begin{equation}\label{equivalentconj}
 \xi(T) \leq \left| V(T) \right| + C(M) 
\end{equation}
 holds. Explicitly, by letting 
 \begin{equation}\label{defineKM}
 K_{M} := 18 \big( 1 + |\chi(M)| \big), 
\end{equation}
 one may then choose $C(M)$ as 
\begin{equation}\label{defineCM} 
 C(M) := K_{M}^{2} + 3 K_{M}. 
\end{equation} 
\end{theorem}

\begin{proof}
 For $v \in V(T)$, we write $d(v)$ in place of the number of triangular faces incident with $v$. The link of $v$ (recalling \eqref{definelk}) is 
 necessarily a simplicial circle, and this follows from $T$ being a simplicial triangulation of a (connected) 
 closed 2-manifold. We thus find that $d(v)$ is 
 the number of neighbors of $v$ in $T^{(1)}$ (recalling \eqref{defineskeleton}), 
 i.e., so that 
\begin{equation}\label{crudedv} 
 d(v) \leq \left| V(T) \right| - 1. 
\end{equation} 

 Let $u$ and $v$ denote distinct vertices in $V(T)$. 
 If a face $\sigma \in F(T)$ contains each of $u$ and $v$, 
 then the closure property associated with simplicial complexes gives us that 
 $\{ u, v \} \in E(T)$. From the assumption that $T$ triangulates $M$, and since $M$ is a (connected) closed 2-manifold (and thus does 
 not have a boundary), 
 every edge is necessarily contained in precisely two faces, i.e., so that
 \begin{equation*}
 \left| \{ \sigma \in F(T) : \{ u, v \} \subseteq \sigma \} \right| = \begin{cases} 
 2, & \text{if $\{ u, v \} \in E(T)$;} \\ 
 0, & \text{otherwise. } 
 \end{cases} 
\end{equation*}
 Since each edge in $E(T)$ is incident with precisely two faces in $F(T)$, 
 the equality 
\begin{equation}\label{3FT2ET} 
 3 |F(T)| = 2 |E(T)| 
\end{equation}
 holds. Write $\chi = \chi(M)$. From \eqref{3FT2ET} together with the 
 formulation of the Euler--Poincar\'{e} formula in \eqref{EulerPoincare}, 
 we obtain that 
\begin{equation}\label{ET3VTchi} 
 \left| E(T) \right| = 3 \left( \left| V(T) \right| - \chi \right). 
\end{equation}
 By the handshaking lemma (in the usual graph-theoretic sense) 
 together with \eqref{3FT2ET} and \eqref{ET3VTchi}, we find that 
\begin{equation}\label{fromhandshake} 
 \sum_{v \in V(T)} d(v) = 3 \left| F(T) \right| = 6 \left( \left| V(T) \right| - \chi \right). 
\end{equation}
 Define 
\begin{equation}\label{defineHT} 
 H(T) := \left\{ v \in V(T) : d(v) > \frac{ \left| V(T) \right| }{3} \right\}. 
\end{equation}
 From \eqref{fromhandshake} together with the definition on display in \eqref{defineHT}, we find (being mindful for the case whereby
 $H(T) = \varnothing$) that 
\begin{equation}\label{boundhandshake} 
 \left| H(T) \right| \frac{ \left| V(T) \right| }{3} 
 \leq \sum_{v \in H(T)} d(v) \leq 6 \left( \left| V(T) \right| - \chi \right). 
\end{equation}
 In turn, the bounds in \eqref{boundhandshake} 
 give us that 
\begin{equation}\label{boundHT} 
 \left| H(T) \right| \leq 18 \left( 1 - \frac{\chi}{ \left| V(T) \right| } \right). 
\end{equation}
 Since $\left| V(T) \right| \geq 1$, we see that 
 $1 - \frac{\chi}{ \left| V(T) \right| } \leq 1 + \left| \chi \right|$. 
 So, by defining $K_{M}$ as in \eqref{defineKM}, 
 the relation \eqref{boundHT} gives us that 
\begin{equation}\label{HTleqKM} 
 \left| H(T) \right| \leq K_{M}.
\end{equation} 
 So, the number of vertices of degree greater than $\frac{\left| V(T) \right|}{3}$
 is bounded above by a constant depending only on $M$. 

 Define $C(M)$ as in \eqref{defineCM}. 
 Our goal, at this point, is to show that $\left| V(T) \right| + C(M)$
 colors suffice, i.e., in order for there to 
 be a cyclic coloration of $T$. 

 We write $\mathcal{E}$ in place of the set consisting of elements $f \in F(T)$ containing at least two vertices in $H(T)$. We then find that 
 there are at most $\binom{ \left| H(T) \right| }{2}$ 
 sets $\{ u, v \}$ of distinct vertices $u, v \in H(T)$
 such that $\{ u, v \}$ is contained in a face in $\mathcal{E}$, 
 and that an arbitrary 2-set $\{ u, v \}$ satisfying the given conditions
 is contained in at most two faces. Consequently, the relation 
\begin{equation}\label{double2sets} 
 \left| \mathcal{E} \right| \leq 2 \binom{ \left| H(T) \right| }{2} 
\end{equation}
 holds, so that \eqref{HTleqKM} and \eqref{double2sets} together give us that 
\begin{equation}\label{boundexceptional} 
 \left| \mathcal{E} \right| \leq K_{M} \left( K_{M} - 1 \right). 
\end{equation}
 We then assign $\left| \mathcal{E} \right|$
 distinct colors to the faces in $\mathcal{E}$. 

 Since we have colored the faces in $F(T)$ containing at least two vertices in $H(T)$, 
 this leads us to consider the faces in $F(T)$ containing exactly one vertex in $H(T)$.
 In this direction, for $x \in H(T)$, we define
\begin{equation}\label{definecalF}
 \mathcal{F}_{x} := \{ \sigma \in F(T) : \sigma \cap H(T) = \{ x \} \}. 
\end{equation}
 
 If $H(T) = \varnothing$, the following stage is vacuous; otherwise, we endow $H(T)$ with a fixed (but arbitrary) linear order relation 
 $\triangleleft$, letting the vertices of $H(T)$ be written as $x_{1}$, $x_{2}$, $\ldots$, $x_{|H(T)|}$
 and ordered so that 
\begin{equation}\label{displayorderH} 
 x_{1} \triangleleft x_{2} \triangleleft \ldots \triangleleft x_{|H(T)|}.
\end{equation}
 Our strategy, at this point, is to 
 color the faces in the families 
 $\mathcal{F}_{x_{1}}$, $\mathcal{F}_{x_{2}}$, $\ldots$, $\mathcal{F}_{x_{|H(T)|}}$
 according to the ordering in \eqref{displayorderH}, 
 so that all of the faces within $\mathcal{F}_{x_{i}}$ are colored in a fixed but arbitrary order
 for fixed $i \in \{ 1, 2, \ldots, |H(T)| \}$. 

 According to the above process, suppose that 
 the face 
\begin{equation}\label{definesigma} 
 \sigma = \{ x_{i}, a, b \} \in \mathcal{F}_{x_{i}}
\end{equation}
 is to be colored, for some index 
 $i \in \{ 1, 2, \ldots, |H(T)| \}$ 
 and some (distinct) vertices $a$ and $b$, with the understanding that all preceding faces have been colored.
 Observe that: By construction, the vertex $x_i$ is the unique vertex that is both in $\{ x_i, a, b \}$ and in $H(T)$. 
 Now, consider previously colored faces (prior to $\sigma$ being colored) 
  that are in the same family $\mathcal{F}_{x_{i}}$ as $\sigma$. 
 From the definition in \eqref{definecalF}, 
 we find that each face in $\mathcal{F}_{x_{i}}$ necessarily contains $x_{i}$. 
 So, the number of previously colored faces
 in $\mathcal{F}_{x_{i}}$
 is bounded above by 
\begin{equation}\label{boundprevious} 
 \left| \mathcal{F}_{x_{i}} \right| - 1 \leq d(x_i) - 1. 
\end{equation} 
 From \eqref{crudedv} and \eqref{boundprevious} together,  we then find that the number of previously colored faces in 
  $\mathcal{F}_{x_{i}}$ is bounded above by 
\begin{equation}\label{upperVTminus2}
 d(x_{i}) - 1 \leq \left| V(T) \right| - 2.
\end{equation}
 We also consider the possibility that previously colored faces may be in the exceptional set $\mathcal{E}$, recalling the upper bound on 
 display in \eqref{boundexceptional}. Moreover, we need to consider, as below, faces in previously colored families. 

 We remain consistent with our notation for $\sigma$, as defined in \eqref{definesigma}, with $i$ as in \eqref{definesigma}. Now, let
 $j < i$. Let $\tau \in \mathcal{F}_{x_{j}}$ be a face having a common vertex with $\sigma$. 
 Since $\tau \in \mathcal{F}_{x_j}$ we have $\tau \cap H(T) = \{ x_{j} \}$.
 In particular, we find that $x_i \not\in \tau$. 
 Hence, if $\tau$ intersects $\sigma = \{ x_i, a, b \}$, 
 then $\tau$ contains $a$ or $b$. 

 We then see that at most two faces in $F(T)$ contain both $a$ and $x_j$ and that at most two faces in $F(T)$ contain both $b$ and $x_j$. 
 Consequently, for each previously processed vertex $x_j$ in $H(T)$, at most four faces in $\mathcal{F}_{x_j}$ can have a common vertex 
 with $\sigma$. There are at most $|H(T)| - 1$ previously processed vertices in $H(T)$, 
 and \eqref{HTleqKM} then gives us that there are at most $K_M$ previously processed vertices in $H(T)$. 
 So, out of the families among $\mathcal{F}_{x_1}$, $\mathcal{F}_{x_2}$, $\ldots$, $\mathcal{F}_{x_{i-1}}$, 
 there are at most $4 K_M$ previously colored faces 
 sharing a common vertex with $\sigma$. 
 This, together with foregoing considerations (recalling \eqref{boundexceptional} and \eqref{upperVTminus2}), 
 allows us to conclude that the number of previously colored faces that cannot have the same color as $\sigma$
 is at most 
 $\left| V(T) \right| - 2 + K_{M} ( K_{M} - 1 ) + 4 K_M$. 
 From \eqref{defineCM}, this value is equal to 
 $\left| V(T) \right| - 2 + C(M)$. 

 So, according to the above procedure, at most $\left| V(T) \right| - 2 + C(M)$ colors are not permitted when $\sigma $ is being colored. 
 So, with a palette of $\left| V(T) \right| + C(M)$ colors, one may always color $\sigma$, without obstructions. Applying the greedy 
 process described above, 
 we may color every face containing exactly one vertex in $H(T)$. 
 
 It remains to consider elements in $F(T)$ containing no vertex in $H(T)$. Let these remaining faces be ordered in a fixed (but 
 arbitrary) way. For such a face $\sigma = \{ a, b, c \}$ currently being colored, 
 we find that: Since none of the vertices in $\{ a, b, c \}$ is in $H(T)$, 
 we have that each of $d(a)$, $d(b)$, and $d(c)$ is bounded above by $\frac{ \left| V(T) \right| }{3}$. 
 The number of faces other than $\sigma$ having a common vertex with $\sigma$ is at most 
 $\big( d(a) - 1 \big) + \big( d(b) - 1 \big) + \big( d(c) - 1 \big)$, 
 which, in turn, is at most $\left| V(T) \right| - 3$, from the given bounds on $d(a)$, $d(b)$, and $d(c)$. 
 So, regardless of the number of faces that have already been colored, 
 at most $\left| V(T) \right| - 3$ colors are forbidden while $\sigma$ is being colored. 
 Since this value is strictly less than $\left| V(T) \right| + C(M)$, there is at least one color available for $\sigma$. 
 
 So, after completing the above procedure, all of the faces have been colored, and, by our above construction, any two faces sharing a 
 common vertex are colored differently. So, our coloring is cyclic and uses at most $\left| V(T) \right| + C(M)$ colors, and hence the 
 upper bound in \eqref{equivalentconj}. 
\end{proof}

\section{Disproof of Chen and Lawrencenko's first conjecture}\label{secdisproof}
 Our simplicial complex that we apply to obtain a counterexample to the first Chen--Lawrencenko 
 conjecture \cite[Conjecture 1]{ChenLawrencenko1999}, as in the proof of Theorem \ref{theoremdisproof} below, is classified 
 as 323-A according to a classification system given 
 in Cervone's work on simplicial immersions of the {K}lein bottle \cite{Cervone1994}. 
 As below, a triangulation $T$ of (a closed, connected $2$-manifold) $M$
 is said to be \emph{vertex-minimal} if
 $\left| V(T) \right| = V_{\min}(M)$. 

\begin{theorem}\label{theoremdisproof}
 There exists a vertex-minimal triangulation $T$ of the Klein bottle such that $|V(T)| = 8$ and $\xi(T) = 9$, i.e., so that the first 
 Chen--Lawrencenko conjecture is false. 
\end{theorem}

\begin{proof}
 Let $T$ be the (unique) finite simplicial complex such that $V(T) = \{ 1, 2, \ldots, 8 \}$ and such that its set of facets is equal to its set 
 of faces, with 

 \ 

\noindent $F(T)$ $=$ $ \{ 123$, \ $124$, \ $135$, \ $146$, \ 
 $156$, \ $236$, \ $245$, \ $257$, \ $268$, \ $278$, \ $345$, 
 \ $348$, \ $367$, \ $378$, \ $468$, \ $ 567 \}$.

 \ 

 \noindent We write $D$ in place of the simplicial 2-disk illustrated with the triangulated copy of $[0, 1]^2$ shown in Figure 
 \ref{uncolored}, i.e., with the understanding that we identify the geometric realization of $D$ 
 with $[0, 1]^2$, i.e., up to homeomorphism, with $[0, 1]^2$ being endowed with its topology inherited from the Euclidean plane. 
 Repeated numerical labels on the boundary 
 denote distinct vertices of $D$ (with the same labels) 
 and are meant to illustrate the quotient to be applied below. 

\begin{figure}
\begin{center}
\includegraphics[scale=0.10]{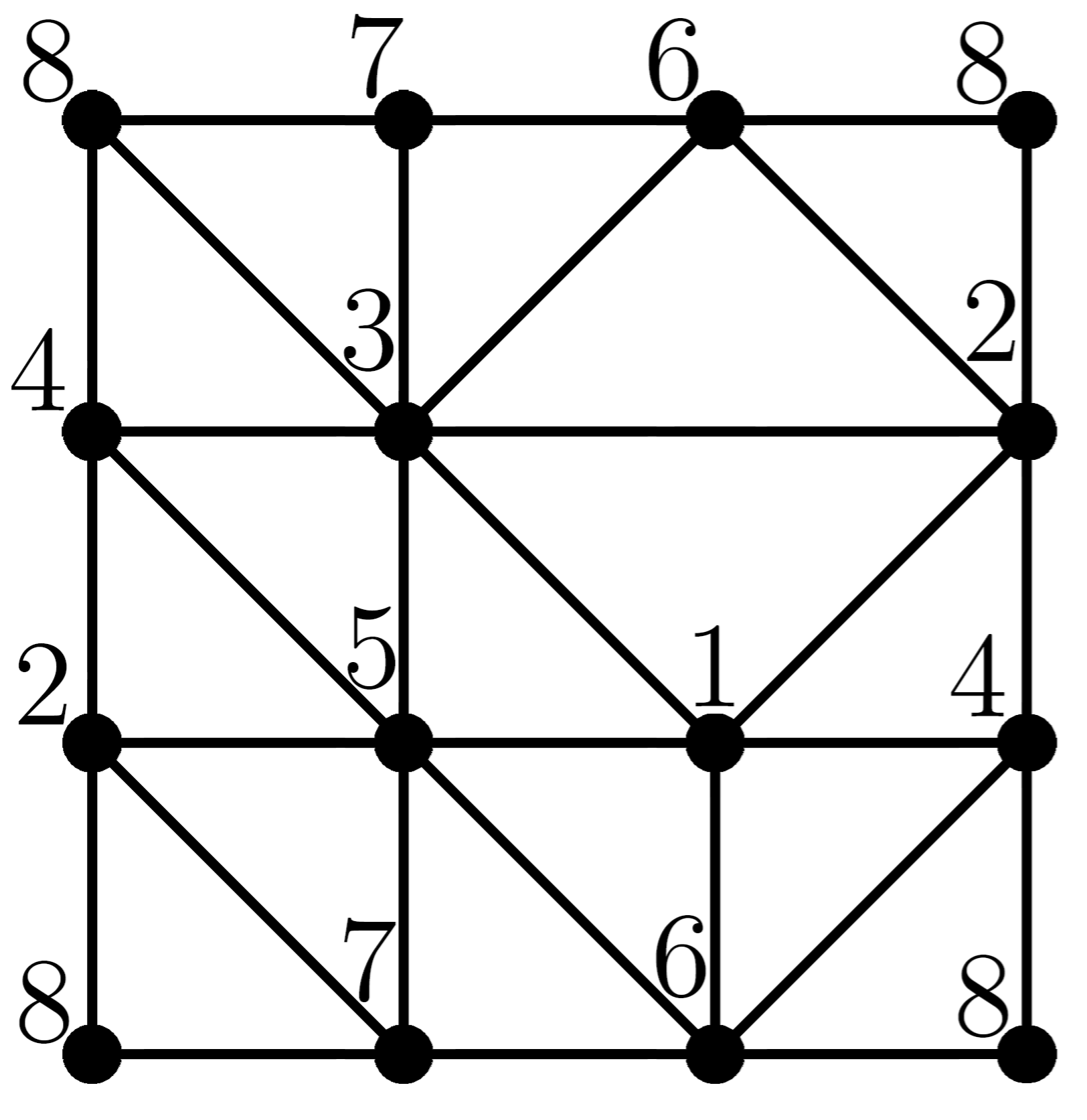}
\caption{\label{uncolored} A minimal triangulation of the Klein bottle. 
 This triangulation is equivalent to Cervone’s triangulation classified as 323-A \cite{Cervone1994}.}
\end{center}
\end{figure}

 Define the equivalence relation $\sim$ on $[0, 1]^2$ 
 so that $(x,0) \sim (x,1)$ and $(0, y) \sim (1, 1-y)$ for $x, y \in [0, 1]$. 
 Also, write $\leftrightarrowtriangle$ in place of the relation given by the adjacency of vertices in a simple graph. 
 With respect to $\sim$, the upper boundary path given by the sequence 
\begin{equation}\label{upperpath}
 8 \leftrightarrowtriangle 
 7 \leftrightarrowtriangle 6 \leftrightarrowtriangle 8
\end{equation}
 of vertices, from left to right, is then identified, in the same order specified in \eqref{upperpath}, with the lower boundary path consisting 
 of the four vertices labeled from left to right with $8$, $7$, $6$, and $8$. In contrast, the left-hand boundary path   
 $ 8 \leftrightarrowtriangle 4 \leftrightarrowtriangle 2 \leftrightarrowtriangle 8$,     from top to bottom, is then identified with the right  
  boundary in a reversed direction, i.e., from bottom to top. 
 These side-pairing maps are simplicial. 
 Also, we find that the quotient $[0, 1]^2/\sim$, endowed with its quotient topology, 
 is homeomorphic to the Klein bottle. 

 With respect to the quotient map indicated above, the $16$ triangular regions indicated in Figure \ref{uncolored}
 yield $16$ distinct $2$-simplices, and the set of vertex sets of these $2$-simplices is precisely $F(T)$. 
 Since no triangle has two vertices identified with respect to $\sim$, and since images of distinct triangles intersect 
 only via common faces, we find that $|T| \cong [0, 1]^2 / \sim$, 
 i.e., so that $T$ is a finite simplicial triangulation of the Klein bottle. 

 So, we have established that the triangulation constructed above and shown in Figure \ref{uncolored} is a triangulation of the Klein 
 bottle $K$. This gives us that $V_{\min}(K) \leq 8$. 
 
 Now, let $S$ be any finite abstract simplicial triangulation of the Klein bottle. Since each triangular face in $F(S)$ contains exactly three 
 edges, and since each edge of an arbitrary triangulation of a connected closed 2-manifold is necessarily contained in precisely two faces, 
 we may verify that the number of incident pairs $(\alpha, \sigma) \in E(S) \times F(S)$ satisfying $\alpha \subseteq \sigma$ gives us 
 the equality 
\begin{equation}\label{countpairs} 
 3 \left| F(S) \right| = 2 \left| E(S) \right|. 
\end{equation}
 From \eqref{countpairs}, together with the vanishing of $\chi(K)$, and together with the definition of the Euler--Poincar\'{e}
 characteristic shown in \eqref{EulerPoincare}, we find that 
\begin{equation}\label{ES3VS} 
 \left| E(S) \right| = 3 \left| V(S) \right|. 
\end{equation} 
 Since $S$ is simplicial, the 1-skeleton of $S$ forms a simple graph, so that 
 $ \left| E(S) \right| \leq \binom{ \left| V(S) \right| }{2}$, 
 and this, together with 
 \eqref{ES3VS}, gives us that 
\begin{equation}\label{binomialupper} 
 3 \left| V(S) \right| \leq \frac{ \left| V(S) \right| \left( \left| V(S) \right| - 1 \right) }{2}. 
\end{equation}
 From \eqref{binomialupper}, it follows that 
\begin{equation}\label{VSgeq7} 
 \left| V(S) \right| \geq 7. 
\end{equation} 
 If $\left| V(S) \right|$ were equal to $7$, then, from \eqref{ES3VS}, we would have that $\left| E(S) \right| = \binom{ 7 }{2}$, so that 
 $S^{(1)}$ would be equivalent to $K_7$. According to Franklin's six-color theorem for the Klein bottle \cite{Franklin1934}, each graph 
 embeddable in the Klein bottle is properly vertex-colorable with at most six colors. 
 In contrast, we find that $K_{7}$ has a chromatic number equal to $7$, 
 and hence is not embeddable in the Klein bottle. 
 This, along with \eqref{VSgeq7}, gives us that $\left| V(S) \right| > 7$. 

 So, we have shown that each simplicial triangulation of the Klein bottle has at least eight vertices. Since $T$, as defined above and as 
 illustrated in Figure \ref{uncolored}, has eight vertices, we find that $T$ is a vertex-minimal triangulation. 

 So, it remains to prove that $\xi(T) = 9$. 
 
 We may verify (being mindful of the boundary identifications indicated above) that the coloring illustrated in 
 Figure \ref{colored} is a cyclic coloration of $T$. 
 This gives us that 
\begin{equation}\label{xiTleq9}
 \xi(T) \leq 9. 
\end{equation} 

\begin{figure}
\begin{center}
\includegraphics[scale=0.10]{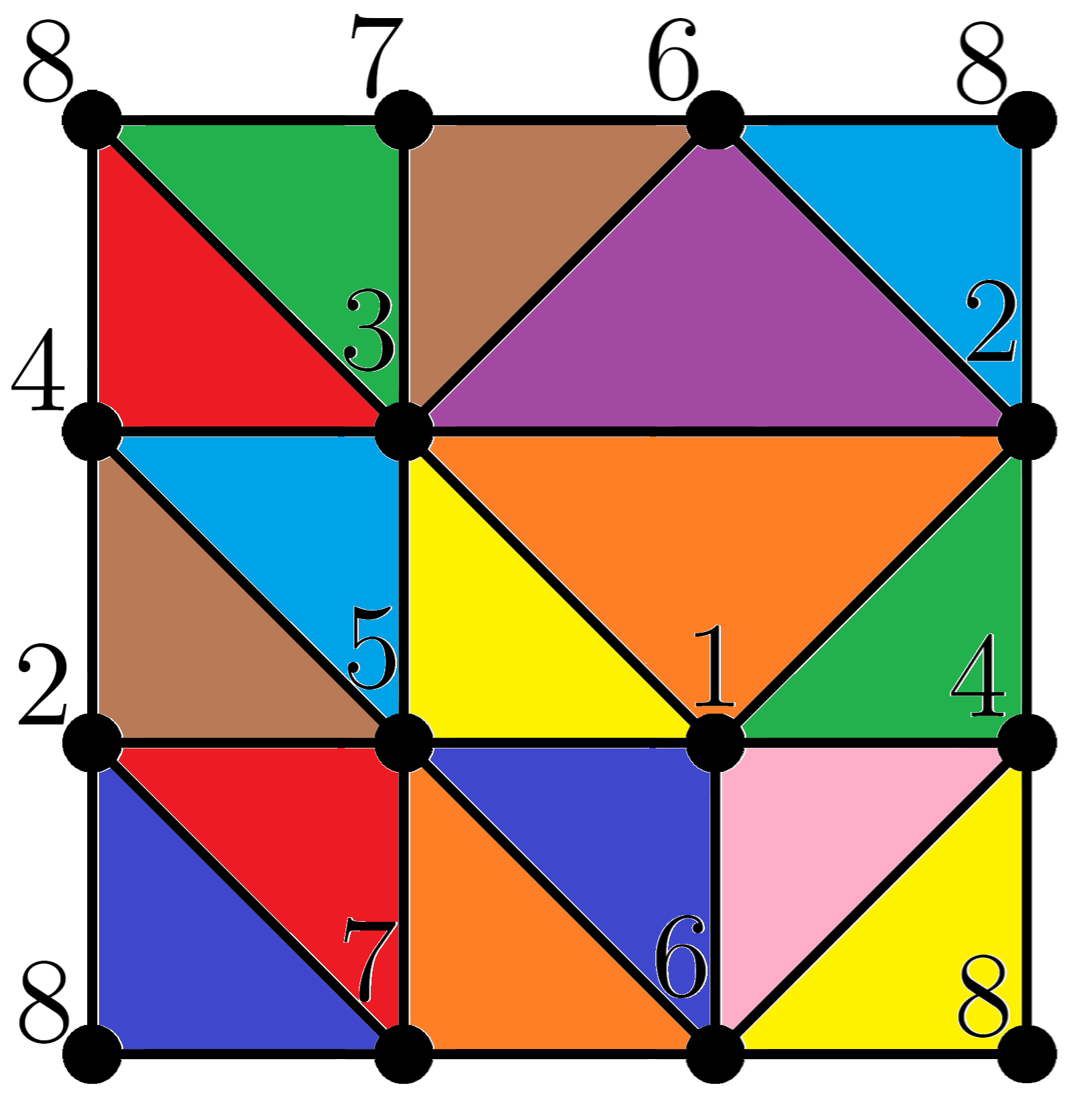}
\caption{\label{colored} A cyclic coloration of a triangulation of the Klein bottle (cf.\ \cite{Cervone1994}).}
\end{center}
\end{figure}

 Now, let 
\begin{equation}\label{displaykappa}
 \kappa\colon F(T) \to \{ 1, 2, \ldots, k \} 
\end{equation}
 be an arbitrary cyclic coloration of $T$, for some positive integer $k$. By definition of a cyclic coloration, if two unequal faces are given 
 the same color, then these faces cannot have a vertex in common. So, each color class associated with the mapping in 
 \eqref{displaykappa} consists of pairwise vertex-disjoint triangular faces. Consequently, since $T$ has $8$ vertices, no color class can 
 contain three faces. So, we find that 
\begin{equation}\label{cardfiber} 
 \left| \kappa^{-1}(i) \right| \leq 2 
\end{equation}
 for each color $i$ in the codomain in \eqref{displaykappa}. 

 We may verify, again with references to Figures \ref{uncolored} and \ref{colored}, that the face $236 \in F(T)$ has a vertex in common with 
 every other face of $T$. So, the color given to $236$ cannot be given to any other face, i.e., so that the color class containing $236$ is a 
 singleton set. This, together with \eqref{cardfiber}, 
 gives us that 
\begin{equation}\label{kappasum} 
 16 = \left| F(T) \right| = \sum_{i=1}^{k} \left| \kappa^{-1}(i) \right| \leq 1 + 2(k-1). 
\end{equation} 
 From \eqref{kappasum}, we find that $k \geq 9$. So, since we let $\kappa$ be an arbitrary cyclic coloration of $T$, we can conclude 
 that $\xi(T) \geq 9$, so that  this,  together with the reverse inequality in \eqref{xiTleq9}, gives   us the desired result. 
\end{proof}

\section{Conclusion}
 As noted by Chen and Lawrencenko \cite{ChenLawrencenko1999}, a proof of Conjecture 2 in their paper would imply that 	
\begin{equation}\label{limsup1} 
 \limsup_{\left| V(T) \right| \to \infty} \frac{\xi(T)}{\left| V(T) \right|} = 1, 
\end{equation}
 with the understanding that the limiting operation on the left of \eqref{limsup1} is taken over all triangulations $T$ of a connected closed 
 2-manifold $M$. Since we have proved the second Chen--Lawrencenko conjecture in the affirmative, 
 we have that \eqref{limsup1} indeed holds true. 

 Chen and Lawrencenko \cite{ChenLawrencenko1999} included two further conjectures in their paper, in addition to the conjectures
 proved/disproved above, 
 which were the main conjectures in their paper highlighted in their abstract.
 We leave it as an open problem to prove the remaining Chen--Lawrencenko conjectures. 

\subsection*{Acknowledgements}
 The author acknowledges extensive interactions with GPT-5.6 Pro during the exploratory and proof-development stages of this work. All 
 AI-generated suggestions were substantially revised, corrected, and independently verified by the author, who assumes full responsibility 
 for the mathematical content. The author's use of AI is fully in accordance with editorial standards on the responsible, ethical, and 
 transparent use of AI. 

\bibliographystyle{plain}
\bibliography{seprefe}

\end{document}